\documentclass[11pt]{article}
\usepackage[english]{babel}
\usepackage{epsfig}
\usepackage{amsmath}
\usepackage{amsthm}
\usepackage{amssymb}

\newtheorem{theorem}{Theorem}[section]
\newtheorem{definition}[theorem]{Definition}

\newtheorem{lemma}[theorem]{Lemma}
\newtheorem{proposition}[theorem]{Proposition}

\newtheorem{remark}[theorem]{Remark}

\title{\bf Differential topological aspects in octonionic monogenic function theory} 

\author{
Rolf S\"{o}ren Krau{\ss}har\\
Fachbereich Mathematik\\
Erziehungswissenschaftliche Fakult\"at\\
Universit\"at Erfurt\\
Nordh\"auser Str. 63\\
99089 Erfurt, Germany\\
soeren.krausshar@uni-erfurt.de }

\begin{document}
\maketitle
\begin{abstract}  
	In this paper we apply a homologous version of the Cauchy integral formula for octonionic monogenic functions to introduce for this class of functions the notion of multiplicity of zeroes and $a$-points in the sense of the topological mapping degree. As a big novelty we also address the case of zeroes lying on certain classes of compact zero varieties. This case has not even been studied in the associative Clifford analysis setting so far.  
	We also prove an argument principle for octonionic monogenic functions for  isolated zeroes and for non-isolated compact zero sets. In the isolated case we can use this tool to prove a generalized octonionic Rouch\'e's theorem by a homotopic argument. As an application we set up a generalized version of Hurwitz theorem which is also a novelty even for the Clifford analysis case. 
	     
\end{abstract}
\noindent {\bf Mathematical Review Classification numbers}: 30G35\\

\section{Introduction}  

Especially during the last three years one notices a significant further boost of interest in octonionic analysis both from mathematicians and from theoretical physicists, see for instance \cite{JRS,KO2018,KO2019,Kra2019-1,Nolder2018}. 
In fact, many physicists currently believe that the octonions provide the adequate setting to describe the symmetries arising in a possible unified world theory combining the standard model of particle physics and aspects of supergravity. See also \cite{Kra2019-2} for the references therein. 

\par\medskip\par

Already during the 1970s, but particularly in the first decade of this century, a lot  of effort has been made to carry over fundamental tools from Clifford analysis to the non-associative octonionic setting. 

Many analogues of important theorems from Clifford analysis could also be established in the non-associative setting, such as for instance a Cauchy integral formula or Taylor and Laurent series representations involving direct analogues of the Fueter polynomials, see for example \cite{Imaeda,Nono,XL2000,XL2001,XL2002,XZL}. Of course, one carefully has to put parenthesis in order to take care of the non-associative nature. 

Although some of these fundamental theorems formally look very similar to those in the associative Clifford algebra setting, Clifford analysis and octononic analysis are two different function theories.  

In \cite{KO2018,KO2019} the authors describe a number of substantial and structural   differences between the set of Clifford monogenic functions from $\mathbb{R}^8 \to Cl_8 \cong  \mathbb{R}^{128}$ and the set of octonionic monogenic functions from $\mathbb{O} \to \mathbb{O}$. This is not only reflected in the different mapping property, but also in the fact that unlike in the Clifford case, left octonionic monogenic functions do not form an octonionic right module anymore. 

The fact that one cannot interchange the parenthesis arbitrarily in a product of octonionic expressions does not permit to carry over a number of standard arguments from the Clifford analysis setting to the octonionic setting. 

In this paper we depart from the octonionic Cauchy integral formula for left or right octonionic monogenic functions, taking special care of the non-associativity by bracketing the terms together in a particular way. First we derive a topological generalized version of this Cauchy integral formula involving the winding number of $7$-dimensional hypersurfaces in the sense of the Kronecker index. From the physical point of view this winding number represents the fourth Chern number of the $G_2$-principal bundles that arise in the application of a generalization of 't Hoofd  ansatz to construct special solutions of generalized $G_2$-Yang-Mills gauge fields, see \cite{Burdik,GTBook}.

This homological version of Cauchy's integral formula is the starting point to introduce first the notion of the order of an isolated zero, or more generally, of an isolated $a$-point of a left (right) octonionic monogenic function. This notion of the order represents the topological mapping degree counting how often the image of a small sphere around zero (or around an arbitrary point $a$) wraps around zero (or $a$, respectively). An application of the transformation formula then leads to an explicit argument principle for isolated zeroes and $a$-points of octonionic monogenic functions. On the one-hand this argument principle naturally relates the fundamental solution of the octonionic Cauchy-Riemann equation with the fourth Chern number of the $G_2$-principal bundles that are related to special solutions of the $G_2$-Yang-Mills equation from 't Hoofd' ansatz. However, this topic will be investigated in detail in one of our follow-up papers. 

On the other hand this argument principle allows us to establish a generalization of Rouch\'e's theorem using a classical homotopy argument. 

In turn, this version of Rouch\'e's theorem enables us to prove that the limit function of a normally convergent sequence of octonionic monogenic functions that have no isolated $a$-points inside an octonionic domain either vanishes identically over the whole domain or it satisfies $\sum_{c \in C}{\rm ord}(f;c)=0$. Note that this statement is slightly weaker than the classical Hurwitz theorem, because in the higher dimensional cases the condition ${\rm ord}(f;c)=0$  does not immediately mean that $f(c)\neq 0$. It is a sufficient but not necessary condition for being zero-free. Anyway, this statement is also new for the associative Clifford analysis setting, of course one has to restrict oneself to paravector-valued functions when addressing this case. 

A big goal and novelty of this paper consists in addressing also the context of  non-isolated zeroes and $a$-points which lie on special simply-connected compact manifolds of dimension $k \in \{1,\ldots,6\}$. Instead of taking small spheres,  the adequate geometric tool is the use of tubular domains that surround  these zero or $a$-point varieties. This geometrical setting allows us to introduce the winding number of a surface wrapping around such a compact zero or $a$-point variety and gives a meaningful definition for the order of a zero variety of an octonionic monogenic function. We also manage to establish an argument principle for these classes of non-isolated zero varieties. These results are even new for the associative Clifford analysis setting and can also be applied to left and right monogenic paravector valued functions in $\mathbb{R}^{n+1}$ for general dimensions $n \in \mathbb{N}$.

To finish we would like to mention that octonions also offer an alternative function theory of octonionic slice-regular functions, see for example \cite{GPzeroes,GP,JRS}.There are of course also connections between octonionic slice-regular functions and octonionic solutions of the generalized octonionic Cauchy-Riemann equations. In the slice-regular context one even gets explicit relations between poles and zeroes as well as a simpler classification of zeroes in a very general situation. In the slice-regular setting only isolated and spherical zeroes can appear and their multiplicity can simply be described in terms of a power exponent appearing in a factorization that makes use of the so-called slice-product. This is a very prosperous direction for developing further powerful function theoretical tools to address problems in the octonionic setting. Note that slice-regular functions also are connected with concrete physical applications, see for instance \cite{Burdik}, in particular also in the construction of special solutions of 't Hoofd ansatz of $G_2$-Yang-Mills solutions.   

However, in this paper we entirely restrict ourselves to solutions of the octonionic Cauchy-Riemann equation, but it is an interesting challenge to pay more attention to  these topics in the framework of other octonionic generalized function theories.

 \section{Basic notions of octonions}
 
 The octonions form an eight-dimensional real non-associative normed division algebra over the real numbers. They serve as a confortable number system to describe the symmetries in recent unifying physical models connecting the standard model of particle physics and supergravity,  see \cite{Burdik,G}. 

Following \cite{Baez,WarrenDSmith} and others, the octonions can be constructed by the usual Cayley-Dickson doubling process. The latter is initiated by taking two pairs of complex numbers $(a,b)$ and $(c,d)$ and forming an addition and multiplication operation by  $$
 (a,b)+(c,d) :=(a+c,b+d),\quad\quad (a,b)\cdot (c,d) := (ac-d\overline{b},\overline{a}d+cb) 
 $$ 
 where $\overline{\cdot}$ denotes the conjugation (anti-)automorphism  which will be extended by $\overline{(a,b)}:=(\overline{a},-b)$ to the set of pairs $(a,b)$. 
 
 In the first step of this doubling procedure we get the real quaternions $\mathbb{H}$. Each quaternion can be written in the form $z=x_0 + x_ 1e_1 + x_2 e_2 + x_3 e_3$ where $e_i^2=-1$ for $i=1,2,3$ and $e_1 e_2 = e_3$, $e_2 e_3 = e_1$, $e_3 e_1 = e_2$ and $e_i e_j = - e_j e_i$ for all mutually  distinct $i,j$ from $\{1,2,3\}$. Already the commutativity has been lost in this first step of the doubling process. However,  $\mathbb{H}$ is still associative.
 
 The next duplification in which one considers pairs of quaternions already leads to the octonions $\mathbb{O}$ which are not even associative anymore. However, in contrast to Clifford algebras, the octonions still form a division algebra. In real coordinates octonions can be expressed in the form  
 $$
 z = x_0 + x_1 e_1 + x_2 e_2 + x_3 e_3 + x_4 e_4 + x_5 e_5 + x_6 e_6 + x_7 e_7
 $$
 where $e_4=e_1 e_2$, $e_5=e_1 e_3$, $e_6= e_2 e_3$ and $e_7 = e_4 e_3 = (e_1 e_2) e_3$. 
 Like in the quaternionic case, we have $e_i^2=-1$ for all $i =1,\ldots,7$ and $e_i e_j = -e_j e_i$ for all mutual distinct $i,j \in \{1,\ldots,7\}$. Their mutual multiplication is illustrated as follows,   
\begin{center}
 \begin{tabular}{|l|rrrrrrr|}
 $\cdot$ & $e_1$&  $e_2$ & $e_3$ & $e_4$ & $e_5$ & $e_6$  & $e_7$ \\ \hline
 $e_1$  &  $-1$ &  $e_4$ & $e_5$ & $-e_2$ &$-e_3$ & $-e_7$ & $e_6$ \\
 $e_2$ &  $-e_4$&   $-1$ & $e_6$ & $e_1$ & $e_7$ & $-e_3$ & $-e_5$ \\
 $e_3$ &  $-e_5$& $-e_6$ & $-1$  & $-e_7$&$e_1$  & $e_2$  & $e_4$ \\
 $e_4$ &  $e_2$ & $-e_1$ & $e_7$ & $-1$  &$-e_6$ & $e_5$  & $-e_3$\\
 $e_5$ &  $e_3$ & $-e_7$ & $-e_1$&  $e_6$&  $-1$ & $-e_4$ & $e_2$ \\
 $e_6$ &  $e_7$ &  $e_3$ & $-e_2$& $-e_5$& $e_4$ & $-1$   & $-e_1$ \\
 $e_7$ & $-e_6$ &  $e_5$ & $-e_4$& $e_3$ & $-e_2$& $e_1$  & $-1$ \\ \hline 	
 \end{tabular}
\end{center}
Fortunately, the octonions still form an alternative and composition algebra. 

In particular, the Moufang rule $(ab)(ca) = a((bc)a)$ holds for all $a,b,c \in \mathbb{O}$. In the special case $c=1$, one obtains the flexibility condition $(ab)a= a(ba)$.  

Let $a = a_0 + \sum\limits_{i=1}^7 a_i e_i$ be an octonion represented with the seven imaginary units as mentioned above. We call $a_0$ the real part of $a$ and write $\Re{a} = a_0$. The conjugation leaves the real part invariant, but $\overline{e_j}=-e_j$ for all $j =1,\ldots,7$. On two general octonions $a,b \in \mathbb{O}$ one has $\overline{a\cdot b} = \overline{b}\cdot \overline{a}$. 

The Euclidean norm and the Euclidean scalar product from $\mathbb{R}^8$ naturally extends to the octonionic case by $\langle a,b \rangle := \sum\limits_{i=0}^7 a_i b_i = \Re\{a \overline{b}\}$ and $|a|:= \sqrt{\langle a,a\rangle} = \sqrt{\sum\limits_{i=0}^7 a_i^2}$. We have the important norm composition property $|a\cdot b|=|a|\cdot|b|$ for all $a,b \in \mathbb{O}$. Every non-zero element $a \in \mathbb{O}$ is invertible with $a^{-1} =\overline{a}/|a|^2$. 

The famous theorems of Frobenius and Hurwitz theorem tell us that 
$\mathbb{R},\mathbb{C},\mathbb{H}$ and $\mathbb{O}$ are the only real normed division algebras.

Further important rules are  
$$
(a\overline{b})b = \overline{b}(ba) =a(\overline{b}b)=a(b \overline{b})
$$  
for all $a,b \in \mathbb{O}$ and, 
$\Re\{b(\overline{a}a)c\} =\Re\{(b \overline{a})(ac)\}$ for all $a,b,c \in \mathbb{O}$, as stated for instance in \cite{CDieckmann} Proposition 1.6. 

We also use the notation $B_8(z,r) :=\{z \in \mathbb{O} \mid |z| < r\}$ and $\overline{B_8(z,r)} :=\{z \in \mathbb{O} \mid |z| \le r\}$ for the eight-dimensional solid open and closed ball of radius $r$ in the octonions and $S_7(z,r)$ for the seven-dimensional sphere $S_7(z,r) :=\{z \in \mathbb{O} \mid |z| = r\}$. If $z=0$ and $r=1$ then we denote the unit ball and unit sphere by $B_8$ and $S_7$, respectively. The notation $\partial B_8(z,r)$ means the same as $S_7(z,r)$.

\section{Argument principle for isolated zeroes of octonionic monogenic functions}
We start this section by recalling the definition of octonionic regularity or octonionic monogenicity in the sense of the Riemann approach. From \cite{Imaeda,XL2000} and elsewhere we quote 
\begin{definition}
	Let $U \subseteq \mathbb{O}$ be an open set. A real differentiable function $f:U \to \mathbb{O}$ is called left (right) octonionic monogenic or equivalently left (right) ${\mathbb{O}}$-regular for short if it satisfies ${\cal{D}} f = 0$ or $f {\cal{D}} = 0$ where $
	{\cal{D}}:= \frac{\partial }{\partial x_0} + \sum\limits_{i=1}^7 e_i \frac{\partial }{\partial x_i}$ stands for the octonionic Cauchy-Riemann operator,  
	where $e_i$ are the octonionic units like defined in the preliminary section before. 
\end{definition}
 In contrast to the associative Clifford analysis setting, the set of left (right) ${\mathbb{O}}$-regular functions do not form an ${\mathbb{O}}$-right (left) module. The following example given in \cite{KO2019} provides a counter-example. Take the function $f(z):= x_1 - x_2 e_4$. A direct computation gives ${\cal{D}}[f(z)] = e_1 - e_2 e_4 = e_1 - e_1 = 0$. But the function $g(z):=(f(z))\cdot e_3 = (x_1 - x_2 e_4) e_3 = x_1 e_3 - x_2 e_7$ satisfies ${\cal{D}}[g(z)] = e_1 e_3 - e_2 e_7 = e_5 -(-e_5) = 2 e_5 \neq 0$. It is clearly the lack of associativity that destroys the modular structure of ${\mathbb{O}}$-regular functions. 
 This is one significant structural difference to Clifford analysis. However, note  that the composition with an arbitrary translation of the form $z \mapsto z + \omega$ where $\omega \in \mathbb{O}$ still preserves monogenicity also in the octonionic case, i.e. ${\cal{D}}f(z+\omega) = 0$ if and only if ${\cal{D}}f (z) = 0$. This is a simple consequence of the chain rule, because the differential remains invariant under an arbitrary octonionic translation.
 
 An important property of left or right ${\mathbb{O}}$-regular functions is that they satisfy the following Cauchy integral theorem, cf. for instance \cite{XL2002}: 
 \begin{proposition} (Cauchy's integral theorem)\\
Let $G \subseteq \mathbb{O}$ be a bounded $8$-dimensional connected star-like domain with an orientable strongly Lipschitz boundary $\partial G$. Let $f \in C^1(\overline{G},\mathbb{O})$. If $f$ is left (resp.) right $\mathbb{O}$-regular inside of $G$, then 
$$
\int\limits_{\partial G} d\sigma(z) f(z) = 0,\quad {\rm resp.}\;\;\int\limits_{\partial G} f(z) d\sigma(z) = 0
$$  	
where $d\sigma(z) = \sum\limits_{i=0}^7 (-1)^j e_i \stackrel{\wedge}{d x_i} = n(z) dS(z)$, where $\stackrel{\wedge}{dx_i} = dx_0 \wedge dx_1 \wedge \cdots dx_{i-1} \wedge dx_{i+1} \cdots \wedge dx_7$ and where $n(z)$ is the outward directed unit normal field at $z \in \partial G$ and $dS =|d \sigma(z)|$ the ordinary scalar surface Lebesgue measure of the $7$-dimensional boundary surface.  
 \end{proposition} 
 
 An important left and right ${\mathbb{O}}$-regular function is the function $q_{\bf 0}: \mathbb{O} \backslash\{0\} \to \mathbb{O},\;q_{\bf 0}(z) := \frac{x_0 - x_1 e_1 - \cdots - x_7 e_7}{(x_0^2+x_1^2+\cdots + x_7^2)^4} = \frac{\overline{z}}{|z|^8}$ whose only singular point is an isolated point singularity of order $7$ at the origin. This function serves as Cauchy kernel in the following Cauchy integral formula for ${\mathbb{O}}$-regular functions. Before we recall this formula, we point out another essential difference to the associative setting: 
 \begin{remark}
 	As already mentioned in {\rm \cite{GTBook}}, in contrast to quaternionic and Clifford analysis, octonionic analysis does {\em not} offer an analogy of a general Borel-Pompeiu formula of the form 
 	$$
 	\int\limits_{\partial G} g(z) d\sigma(z) f(z) = 0,
 	$$
 	not even if $g$ is right $\mathbb{O}$-regular and $f$ left $\mathbb{O}$-regular, independently how we bracket these terms together. The lack of such an identity is again a consequence of the lack of associativity. However, if one of these functions is the Cauchy kernel, then one obtains a generalization. 
 \end{remark}
  For convenience we recall from \cite{Imaeda,Nono,XL2002}:
 \begin{proposition}
Let $U \subseteq \mathbb{O}$ be a non-empty open set and $G \subseteq U$ be an $8$-dimensional compact oriented manifold with a strongly Lipschitz boundary $\partial G$. If $f: U \to \mathbb{O}$ is left (resp. right) $\mathbb{O}$-regular, then for all $z \not\in \partial G$
$$
\chi(z)f(z)= \frac{3}{\pi^4} \int\limits_{\partial G} q_{\bf 0}(w-z) \Big(d\sigma(w) f(w)\Big),\quad\quad \chi(z) f(z)= \frac{3}{\pi^4} \int\limits_{\partial G}   \Big(f(w)d\sigma(w)\Big) q_{\bf 0}(w-z),
$$
where $\chi(z) = 1$ if $z$ is in the interior of $G$ and $\chi(z)=0$ if $z$ in the exterior of $G$. 
 \end{proposition}
Note that the way how the parenthesis are put is very important. Putting the parenthesis in the other way around, would lead in the left $\mathbb{O}$-regular case to a different formula of the form
$$
\frac{3}{\pi^4} \int\limits_{\partial G} \Big( q_{\bf 0}(w-z) d\sigma(w)\Big) f(w) = \chi(z) f(z) + \int\limits_G \sum\limits_{i=0}^7 \Big[q_{\bf 0}(w-z),{\cal{D}}f_i(w),e_i  \Big] dw_0 \wedge \cdots \wedge dw_7, 
$$
where $[a,b,c] := (ab)c - a(bc)$ stands for the associator of three octonionic elements. The volume integral which appears additionally always vanishes in algebras where one has the associativity, such as in Clifford algebras. 
See \cite{XL2002}. 
 
An important subcase is obtained when we take for $f$ the constant function $f(z) = 1$ for all $z \in U$ which is trivially left and right $\mathbb{O}$-regular. Then the Cauchy integral simplifies to the constant value  
$$
\chi(z) = \frac{3}{\pi^4} \int\limits_{\partial G} q_{\bf 0}(w-z) d\sigma(w),\quad {\rm resp.}\; \chi(z)= \frac{3}{\pi^4} \int\limits_{\partial G}   d\sigma(w) q_{\bf 0}(w-z),
$$
simply indicating if $z$ belongs to the interior or to the exterior component of $\partial G$.  
This is the starting point to introduce a following generalized topological version of the above stated Cauchy integral formula. Following for instance \cite{Hempfling} 
one can consider more generally for $G$ a bounded Lipschitz domain whose boundary $\partial G$ could be a $7$-chain, homologous to a differentiable $7$-chain with image $\partial B(z,r)$, parametrized as 
\begin{equation}\label{param}
\partial G = \{x_0(\lambda_1,\ldots,\lambda_7) + \sum\limits_{i=1}^7 x_i(\lambda_1,\ldots,\lambda_7) e_i\}. 
\end{equation}
In this more general case one has 
$$
w_{\partial G}(z) = \frac{3}{\pi^4} \int\limits_{\partial G} q_{\bf 0}(w-z) d\sigma(w),\quad {\rm resp.}\; w_{\partial G}(z)= \frac{3}{\pi^4} \int\limits_{\partial G}   d\sigma(w) q_{\bf 0}(w-z),
$$
where $w_{\partial G}(z)$ represents the topological winding number, sometimes called the Kronecker-index (cf. \cite{ghs}), counting how often $\partial G$ wraps around $z$. Note that this is a purely topological entity induced by
$$
H_8(\partial G,\partial G - z) \cong H_8(B_8,S_7) \cong \tilde{H}_7(S_7) \cong \mathbb{Z},
$$
where  $H_8$ is the related homology group and $\tilde{H}_7$ the related reduced homology group. Due to this property, the winding number $w_{\partial G}(z)$ is always an integer. This is the basis for the more general topogical version of Cauchy's integral formula: 
\begin{theorem}\label{topcauchy}(Topological generalized octonionic Cauchy integral formula)\\	
Let $U \subseteq \mathbb{O}$ be an open set and $G$ be a closed manifold whose boundary $\Gamma$ is a strongly Lipschitz $7$-chain. If $f:U \to \mathbb{O}$ is left $\mathbb{O}$-regular, then we have the identity
$$
w_{\Gamma}(z)f(z)= \frac{3}{\pi^4} \int\limits_{\Gamma} q_{\bf 0}(w-z) \Big(d\sigma(w) f(w)\Big),\quad z \not\in \Gamma
$$
where $w_{\Gamma}(z)$ is the topological winding number counting how often $\Gamma$ wraps around $z$. The latter equals zero if $z$ in a point from the exterior of $G$.
\end{theorem}
\begin{remark}
Note that if we put the parenthesis the other way around, then we get the identity 
$$
\frac{3}{\pi^4} \int\limits_{\partial G} \Big( q_{\bf 0}(w-z) d\sigma(w)\Big) f(w) = w_{\partial G}(z) f(z) + \int\limits_G \sum\limits_{i=0}^7 \Big[q_{\bf 0}(w-z),{\cal{D}}f_i(w),e_i  \Big] dw_0 \wedge \cdots \wedge dw_7.  
$$
The volume integral is not affected in the topological version, because we simply integrate over the volume and orientation does not play any role, because the scalar differential $dV = dw_0 \wedge \cdots \wedge dw_7$ has no orientation.  
\end{remark}
\begin{remark}
Comparing with {\rm \cite{Burdik,GTBook}}, we can relate the octonionic winding number with the fourth Chern number of the $G_2$-principal bundles associated to special solutions of $G_2$ Yang-Mills gauge fields arising in generalizing 't Hoofd ansatz, see {\rm \cite{Burdik,GTBook}}. This allows us to explicitly relate the fundamental solution of the octonionic Cauchy-Riemann equation with Chern numbers of the related $G_2$-principal bundles. We will shed some more light on this interesting connection in a follow-up paper.  
\end{remark}
The topological winding number is also the key tool to define a generalized notion of multiplicity of zeroes and $a$-points of $\mathbb{O}$-regular functions. To proceed to the definition and classification of $a$-points we first need the octonionic identity theorem:
\begin{proposition}\label{identity} 
Let $G \subseteq \mathbb{O}$ be an $8$-dimensional domain. Suppose that $f,g: G \to \mathbb{O}$ are two left (right) $\mathbb{O}$-regular functions. If there exists a $7$-dimensional smooth sub-manifold $V$ where $f(z)=g(z)$ for all $z \in V$, then we have $f(z) = g(z)$ for all $z \in G$.  
\end{proposition}	
In particular, a left (right) $\mathbb{O}$-regular function satisfying $f(z)=0$ on a $7$-dimensional sub-manifold vanishes identically. Similarly, if there is an octonion $a \in \mathbb{O}$ such that $f(z)=a$ for all $z \in V$, then $f(z) = a$ for all $z \in G$. Although the proof only uses basic tools of octonionic analysis, we prefer to present it in detail, as we are not aware of a direct reference in the literature addressing the particular octonionic setting. For the proof of the statement in the associative Clifford analysis setting we refer to \cite{ghs}, p. 187.  	
\begin{proof}
The proof can be done by extending R. Fueter's argumentation from the quaternionic case presented in \cite{Fueter1948-49} on pp.185-189. Without loss of generality we consider the situation where $g(z)$ is the zero function. Suppose now that $V$ is a seven dimensional smooth manifold where $f|_V = 0$. Consider an arbitrary point $c \in V$ with $f(c)=0$. Since $V$ is $7$-dimensional and smooth one can find seven $\mathbb{R}$-linearly independent unit octonions, say ${\bf n}_1, \ldots, {\bf n}_7$ with $|{\bf n}_h|=1$ $(h=1,\ldots,7)$ that lie in the $7$-dimensional tangent space $T_V(c)$. Next define  $\xi^{(h)}_0 := \langle {\bf n}_h,1\rangle$ and $\xi^{(h)}_j :=\langle {\bf n}_h,e_j\rangle$ for $j=1,\ldots,7$ where $\langle\cdot,\cdot\rangle$ is the scalar product on $\mathbb{O}$ defined in Section~2.  
Notice that all the values $\xi^{(h)}_j$ are real for all $j=0,1\ldots,7$ and all $h=1,\ldots,7$. Next consider for each point $c \in V$ the real $7\times8$-matrix composed by the seven rows constituted by the eight real coordinates of the seven octonions ${\bf n}_1,\ldots,{\bf n}_7$, respectively, i.e.
$$
A:=\left(\begin{array}{cccc} \xi^{(1)}_0 & \xi^{(1)}_1 & \cdots & \xi^{(1)}_7\\   
\xi^{(2)}_0 & \xi^{(2)}_1 & \cdots & \xi^{(2)}_7\\
\vdots & \vdots & \cdots & \vdots \\
\xi^{(7)}_0 & \xi^{(7)}_1 & \cdots & \xi^{(7)}_7\\
\end{array}  \right)
$$  
Re-interpreting the seven octonions ${\bf n}_j$ as column vectors from $\mathbb{R}^8$, we have $rank({\bf n}_1,\ldots,{\bf n}_7)=7$ in view of the $\mathbb{R}$-linear independency. Consequently, also the rank of the largest non-vanishing sub-determinant must equal $7$. Without loss of generality we may suppose that 
\begin{equation}\label{domega}
\det\left(\begin{array}{cccc} \xi^{(1)}_1 & \xi^{(1)}_2 & \cdots & \xi^{(1)}_7\\   
\xi^{(2)}_1 & \xi^{(2)}_2 & \cdots & \xi^{(2)}_7\\
\vdots & \vdots & \cdots & \vdots \\
\xi^{(7)}_1 & \xi^{(7)}_2 & \cdots & \xi^{(7)}_7\\
\end{array}  \right) \neq 0.
\end{equation}
Otherwise, we change the labels of the components. 

Next we use that $f(z) = f_0(z) + \sum\limits_{k=0}^7  f_k(z) e_k \equiv 0$ on $V$. Therefore, the directional derivatives also vanish all, i.e. $\frac{\partial f}{\partial {\bf n}_h} = 0$ for each $h=1,2,\ldots,7$. Using the ordinary chain rule gives seven equations:
$$
\frac{\partial f}{\partial {\bf n}_h} = \sum\limits_{k=0}^7 \frac{\partial f}{\partial x_k} \frac{\partial x_k}{\partial {\bf n}_h} = \sum\limits_{k=0}^7 \frac{\partial f}{\partial x_k} \xi^{(h)}_k=0,\quad h=1,\ldots,7.
$$
Additionally, as eighth condition, $f$ has to satisfy the octonionic left Cauchy-Riemann equation $\sum\limits_{k=0}^7 e_k \frac{\partial f}{\partial x_k} = 0$. 

Consider the formal octonionc determinant 
$$
\det(\Omega) := \det\left(\begin{array}{cccc} 1 & e_1 & \cdots & e_7\\   
\xi^{(1)}_0 & \xi^{(1)}_1 & \cdots & \xi^{(1)}_7\\
\vdots & \vdots & \cdots & \vdots \\
\xi^{(7)}_0 & \xi^{(7)}_1 & \cdots & \xi^{(7)}_7\\
\end{array}  \right),
$$
defined formally in the usual way. Note that the non-associativity does not lead to  ambiguous interpretations, because only the entities $e_1,\ldots,e_7$ are octonions, while the other entries $\xi^{(h)}_k$ are all real-valued expressions. So, this formal determinant is a well-defined octonion. 
The eight equations mentioned above could be satisfied under two particular circumstances only. Firstly, they could be satisfied if $\det(\Omega)$ vanished. However, this is impossible. Notice that $\det(\Omega)$ represents an octonion. An octonion only vanishes if {\em all} its real components vanish. However, we obviously have $\Re\{\det(\Omega))\}\neq 0$ in view of (\ref{domega}). The only remaining second option is that 
$$
\frac{\partial f}{\partial x_k} = 0,\quad k=0,1,\ldots,7
$$ 
at each $z \in V$. Note that also the octonionic Cauchy integral formula implies that the left $\mathbb{O}$-regularity of $f$ is also inherited by all partial derivatives of $f$. Consequently, the same argumentation is also true for all partial derivatives $\frac{\partial^{n_1+\cdots+n_7}}{\partial x_1^{n_1} \cdots \partial x_7^{n_7}} f(z) = 0$. Following \cite{Imaeda,XL2001} we can expand $f$ into a Taylor series  around each left $\mathbb{O}$-regular point $z=c \in V$ 
of the form $f(z) = \sum\limits_{n=0}^{\infty} \sum\limits_{n=n_1+\cdots+n_7} V_{\bf n}(z-c) c_{n_1,\ldots,n_7} \equiv 0$ where 
$$
V_{\bf n}(z) = \frac{1}{|{\bf n}|!} \sum\limits_{\pi \in perm({\bf n})} (Z_{\pi(n_1)}(Z_{\pi(n_2)}( \cdots (Z_{\pi(n_{6})} Z_{\pi(n_{7})})\cdots))).
$$
One has to apply the parenthesis in this particular way. Due to the lack of associativity, the parenthesis cannot be neglected. 
Here, $perm({\bf n})$ denotes the set of all distinguishable permutations of the sequence $(n_1,n_2,\ldots,n_7)$ and $Z_i := V_{\tau(i)}(z) := x_i - x_0 e_i$ for all $i=1,\ldots,7$, cf. \cite{XL2001} Theorem C p.208. Here $\tau(i)$ is the multi-index $(n_1,\ldots,n_7)$ where $n_j = 0$ for all $j \neq i$ and $n_i=1$.

However, following also from \cite{XL2001}, $c_{n_1,\ldots,n_7} :=\Bigg( \frac{\partial^{n_1+\cdots+n_7}}{\partial x_1^{n_1} \cdots \partial x_7^{n_7}} f(z)\Bigg)_{z=c} = 0$. The uniqueness of the Taylor series representation implies that $f$ must be identically zero over the whole domain $G$.  
\end{proof}

\begin{remark}
If one considers instead of $\mathbb{O}$-regular functions, the set of slice-regular functions from {\rm \cite{GP,JRS}}, then one even gets a much stronger version of the identity theorem, namely stating that two slice-regular functions already coincide with each other, when they coincide with each other on a one-dimensional set with an accumulation point. This has a strong consequence on the structure of the zeroes.    
\end{remark}

Since also the octonions form a normed algebra, we can introduce the notion of an isolated $a$-point of an $\mathbb{O}$-regular function as follows, compare with \cite{Hempfling,Kra2004}:
\begin{definition}
Let $U \subseteq \mathbb{O}$ be an open set and $f:U \to \mathbb{O}$ be a function. Then we say that $f$ has an isolated $a$-point at $c \in U$, if $f(c)=a$ and if there exists a positive real $\varepsilon > 0$, such that $f(z) \neq a$ for all $z \in B(c,\varepsilon) \backslash\{c\}$. If $a=0$, then we call $c$ an isolated zero.   
\end{definition}
Let $U \subseteq \mathbb{O}$ be an open set, $c \in U$ and $f:U \to \mathbb{O}$ be a real differentiable function, i.e. we suppose that each real component function $$f_i:U \to \mathbb{R}\;\;(i=0,1,\ldots,7)\;\; {\rm of}\;\; f(z)= f_0(z) + f_1(z) e_1 + \cdots + f_7(z)e_7$$ is partial differentiable.   
According to the implicit function theorem in $\mathbb{R}^8$ a sufficient criterion for an $a$-point of a real-differentiable function $f:U \to \mathbb{O}$  of being an isolated $a$-point with $f(c)=a$ is that the Jacobian determinant does not vanish $\det(Jf)(c) := \det\Big(\frac{\partial f_j}{\partial x_j} \Big)_{0 \le i,j \le 7} \neq 0$. However, this clearly is just a sufficient criterion, as the following example illustrates. 
Take for instance the function $:\mathbb{O} \to \mathbb{O}$ defined by 
\begin{eqnarray*}
f(z)&:=& V_{2,0,0,\ldots,0}(z) + V_{0,2,0,\ldots,0}(z)+\cdots+V_{0,\ldots,0,2}(z)\\ &=& Z_1^2+Z_2^2+\cdots+Z_7^2  = (x_1^2+\cdots+x_7^2-7x_0^2) - 2 \sum\limits_{i=1}^7 x_0x_i e_i
\end{eqnarray*}
which is clearly left and right $\mathbb{O}$-regular in the whole algebra $\mathbb{O}$. 
Obviously, one has $f(0)=0$. In general, $f(z)=0$ implies that first 
$$x_1^2+x_2^2+\cdots+x_7^2=7x_0^2$$
and one has $x_0 x_i = 0$ for each $i=1,\ldots,7$. The first relation implies that $x_0 = \pm \frac{1}{\sqrt{7}} \sqrt{x_1^2+\cdots+x_7^2}$. Inserting this expression into the other relations yields $\sqrt{x_1^2+\cdots+x_7^2} x_i = 0$ for all $i=1,\ldots,7$. Since $x_1^2+\cdots+x_7^2 > 0$ whenever $(x_1,x_2,\ldots,x_7) \neq (0,0,\ldots,0)$ we must have $x_i = 0$ for all $i=1,\ldots,7$. Therefore, also $x_0=0$. Summarizing $z=0$ is the only zero of $f$ and therefore it must be an isolated zero. The Jacobian matrix however is:
$$
(Jf)(z):= \left(\begin{array}{cccccc} -14x_0 &  2x_1 & 2x_2    & \cdots & 2x_6    & 2x_7\\
									  -2x_1 & -2x_0 & 0    & \cdots & 0    & 0 \\
									  -2x_2     &   0   &-2x_0 & \cdots &0     & 0\\
									  \vdots&\vdots &\vdots&\vdots  &\vdots& \vdots\\
									  -2x_6     &   0   & 0    &   \cdots & -2x_0 & 0 \\
									  -2x_7 &   0   & 0    &  \cdots & 0 & -2x_0 
									  \end{array} \right).
$$
Inserting $z=0$ yields $\det(Jf)(z)=0$. 

A typical example of a non-linear left $\mathbb{O}$-regular function with one single octonionic isolated zero $z^*$ satisfying $Jf(z^*) \neq 0$ can be constructed by applying T. Hempfling's construction from \cite{Hempfling} p.111.  
Adapting from \cite{Hempfling}, the octonionic version of the function 
$$
f(z)= (x_1 x_2 \cdots x_7-1)-(x_0 x_2 \cdots x_7-1)e_1 - \cdots -(x_0 x_1 \cdots x_6-1)e_7   
$$  
actually is left $\mathbb{O}$-regular. We have 
$$
\frac{\partial f}{\partial x_0} = -\sum\limits_{j=1}^7\Big(\prod\limits_{i\neq 0, i \neq j} x_i \Big) e_j  
$$ 
and for $k \in \{1,\ldots,7\}$ 
$$
e_k \frac{\partial f}{\partial x_k} = \Big(\prod\limits_{i\neq 0, i \neq k}  x_i \Big) e_k -  \Big(\prod\limits_{i\neq 0, i \neq k}  x_i \Big) e_k e_1 - \cdots -  \Big(\prod\limits_{i\neq 0, i \neq k}  x_i \Big) e_k e_7. 
$$
So, $f$ actually satisfies $\frac{\partial f}{\partial x_0} + \sum\limits_{k=1}^7 e_k \frac{\partial f}{\partial x_k} = 0$. 

As one readily observes, one has $f(z^*)=0$ when inserting $z^{*}=1+e_1+\cdots+e_7$. Furthermore, $\frac{\partial f_i}{\partial x_j} = \delta_{ij} \prod\limits_{k=0,k\neq i,j}^7 x_k$, where $\delta_{ij}$ denotes the ordinary Kronecker symbol. Thus,
$$
Jf((1+e_1+\cdots+e_7)) = \left(  \begin{array}{ccccc} 0 & 1 & 1 & \cdots & 1\\
											   1 & 0 & 1 & \cdots & 1 \\
											   1 & 1 & 0 & \cdots & 1\\
											   \vdots & \vdots & \vdots & \vdots & \vdots\\
											   1 & 1 & 1 & \cdots & 0 \end{array}
\right),
$$
and therefore $\det(Jf((1+e_1+\cdots+e_7))) = -7 \neq 0$. $z^*$ clearly is an isolated zero of $f$. 

Note that in general a left or right $\mathbb{O}$-regular function can possess also zeroes that lie on $k$-dimensional manifolds with $k \le 6$. The case $k=7$ cannot appear as direct a consequence of Proposition~\ref{identity}, because if a left $\mathbb{O}$-regular function vanishes on a $7$-dimensional manifold, then it must be identically zero over the whole $8$-dimensional space.  Furthermore, note that the zero sets of left or right $\mathbb{O}$-regular functions must be real analytic manifolds. Already very simple octonionic functions can have connected sets of zeroes. Adapting from \cite{Zoell} and \cite{Hempfling}, in the octonionic case the simplest examples (for each dimension) are 
\begin{eqnarray*}
f(z)=Z_1^2+\cdots+Z_7^2 & & {\rm isolated\;zero\;at}\;z^*=0\\
f(z)=Z_1^2+\cdots+Z_6^2 & & {\rm 1-dimensional\;zero\;set \;at}\;z \in e_7 \mathbb{R}\\
f(z)=Z_1^2+\cdots+Z_5^2 & & {\rm 2-dimensional\;zero\;set \;at}\;z \in e_6 \mathbb{R} \oplus e_7 \mathbb{R}\\
f(z)=Z_1^2+\cdots+Z_4^2 & & {\rm 3-dimensional\;zero\;set \;at}\;z \in e_5 \mathbb{R} \oplus e_6 \mathbb{R} \oplus e_7 \mathbb{R}\\
f(z)=Z_1^2+Z_2^2+Z_3^2 & & {\rm 4-dimensional\;zero\;set \;at}\;z \in e_4 \mathbb{R} \oplus \cdots \oplus e_7 \mathbb{R}\\
f(z)=Z_1^2+Z_2^2 & & {\rm 5-dimensional\;zero\;set \;at}\;z \in e_3 \mathbb{R} \oplus \cdots \oplus e_7 \mathbb{R}\\
f(z)=Z_1^2 & & {\rm 6-dimensional\;zero\;set \;at}\;z \in e_2 \mathbb{R} \oplus \cdots \oplus e_7 \mathbb{R}\\
\end{eqnarray*}
 where $Z_i$ are again the octonionic Fueter polynomials $Z_i = x_i-x_0e_i$ for $i=1,\ldots,7$. 
 
 Generalizing the construction from \cite{Hempfling} a further class of interesting examples can be gained from the following construction. Let $k \in \{2,\ldots,6\}$ be an integer and consider the function $f:\mathbb{O} \to \mathbb{O}$, $f(z):= Z_1^2+\cdots+Z_k^2-\sum\limits_{j=k+1}^7 Z_je_j$ composed by the octonionic Fueter polynomials. Again, this function is both left and right $\mathbb{O}$-regular and can be written in the form 
 $$
 f(z)=\Big(\sum\limits_{i=1}^k x_i^2\Big) - k x_0^2+(7-k)x_0 - 2 x_0\sum\limits_{i=1}^k x_i e_i + \sum\limits_{i=k+1}^7 x_i e_i.
 $$
  when switching to the ordinary variables $x_i$.
 Now consider the function $g(z):=f(z)-R^2$, where $R>0$ is a real number. Then $g(z)=0$ if and only if the following system of equations is satisfied
 \begin{eqnarray*}
 \sum\limits_{i=1}^k x_i^2-k x_0^2-R^2+(7-k)x_0
 &=& 0 \\
 x_0 x_i &=& 0,\;\;i=1,\ldots,k\\
 x_i & = & 0,\;\;i=k+1,\ldots,7.
 \end{eqnarray*}
First case: $x_0=0$. Then $g(z)=0$ if and only if $\sum\limits_{i=1}^k x_i^2-R^2 = 0$. Now the zero variety of $g$ is the compact $k-1$-dimensional sphere of radius $R$ centered around the origin in the subspace generated by $e_1,e_2,\ldots,e_k$. \\
Second case $x_0 \neq 0$. Then $x_i = 0$ for all $i=1,\ldots,7$. In this case $g(z)=0$ if and only if $-kx_0^2-R^2+(7-k)x_0=0$. This condition can only be satisfied if $x_0 = - \frac{1-\frac{7}{k}}{2} \pm \sqrt{\frac{(1-\frac{7}{k})^2}{4} - \frac{R}{k}}$, provided the value in the square root expression is not negative. In this case the zero set consists at most of two isolated points $(x_0,0,\ldots,0)$ on the real axis.   
\par\medskip\par
In the spirit of \cite{Hempfling,HeKra,Kra2004} we now proceed to introduce the order of an isolated zero or isolated $a$-point  of an $\mathbb{O}$-regular function. This can be done like in the quaternionic and Clifford analysis case in terms of the topological Cauchy integral mentioned above and then represents the order of an isolated $a$-point in the sense of the topological mapping degree. 
\begin{definition}\label{isolatedorder}
	Let $U \subseteq \mathbb{O}$ be an open set, $U \neq \emptyset$. Let $f:U \to \mathbb{O}$ be left $\mathbb{O}$-regular (resp. right $\mathbb{O}$-regular) and suppose that $c \in U$ is an isolated $a$-point of $f$, i.e. $f(c)=a$ with $a \in \mathbb{O}$. Choose an $\varepsilon > 0$ such that $\overline{B}(c,\varepsilon) \subseteq U$ and suppose that $f(z) \neq 0$ for all $z \in \overline{B}(c,\varepsilon) \backslash\{c\}$. Then, 
	$$
	{\rm ord}(f-a;c) := \frac{3}{\pi^4} \int\limits_{(f-a)(\partial B(c,\varepsilon))} q_{\bf 0}(w) d\sigma(w),\quad {\rm resp.}\;{\rm ord}(f-a;c) := \frac{3}{\pi^4} \int\limits_{(f-a)(\partial B(c,\varepsilon))} d\sigma(w) q_{\bf 0}(w) 
	$$
	is called the order of the isolated $a$-point of the octonionic left (resp. right) $\mathbb{O}$-regular function $f$ at $c$.
\end{definition}
 In the case where $a=0$, we address the order of isolated zeroes of $f$, which in the left $\mathbb{O}$-regular case equals the Cauchy integral
 $$
 {\rm ord}(f;c) := \frac{3}{\pi^4} \int\limits_{f(\partial B(c,\varepsilon))} q_{\bf 0}(w) d\sigma(w).
 $$
 \begin{proposition}
 The numbers ord$(f-a;c)$ are integers and count how often the image of the sphere around the octonionic $a$-point wraps around $a$ and therefore represents the notion of the order of an $a$-point in the sense of the topological mapping degree.   	
  \end{proposition}
\begin{proof}
The topological generalized version of the octonionic Cauchy integral formula (Theorem~\ref{topcauchy}) tells us that every octonionic function $h:U \to \mathbb{O}$ that is left $\mathbb{O}$ regular over an open set $U$ which contains a closed manifold $G$ whose boundary $\Gamma$ is a strongly Lipschitz $7$-chain satisfies 
$$
w_{\Gamma}(y)h(y) = \frac{3}{\pi^4} \int\limits_{\Gamma} q_{\bf 0}(w-y) \Big(d\sigma(w) h(w)\Big).
$$
So, in the case where $h(z) = 1$ for all $z \in U$, one has 
$$
w_{\Gamma}(y) = \frac{3}{\pi^4} \int\limits_{\Gamma} q_{\bf 0}(w-y) d\sigma(w).
$$  
In view of the mentioned property $H_8(\Gamma,\Gamma - c) \cong \tilde{H}_7(S_7)$ one can replace in the latter equation $\Gamma$ by the homeomorphic equivalent small sphere $\partial B(c,\varepsilon)$, so we have 
$$
w_{\Gamma}(y) = \frac{3}{\pi^4} \int\limits_{\partial B(c,\varepsilon)} q_{\bf 0}(w-y) d\sigma(w).
$$
Next we replace the octonion $y$ by $f(c)-a$ and $\partial B(c,\varepsilon)$ by $(f-a)(\partial B(c,\varepsilon))$ and one obtains 
\begin{eqnarray*}
w_{\Gamma}(y) &=& \frac{3}{\pi^4} \int\limits_{(f-a)\partial B(c,\varepsilon)} q_{\bf 0}(w-(f(c)-a)) d\sigma(w)\\
&= & \frac{3}{\pi^4} \int\limits_{(f-a)\partial B(c,\varepsilon)} q_{\bf 0}(w) d\sigma(w) \\
&=& w_{(f-a)(\partial B(c,\varepsilon))}(0).
\end{eqnarray*}
We recall that also $f(\partial B(c,\varepsilon))$ and hence also the translated expression  $(f-a)(\partial B(c,\varepsilon))$ represents a $7$-dimensional cycle, cf. \cite{AH} p. 470. 
\end{proof}
\begin{remark}\label{zero-order}
In contrast to complex analysis it can happen that one has ord$(f;c) = 0$ even if $f(c)=0$. This can occur for instance when the outward normal field of the surface of the image of the boundary cycle $\Gamma$ turns into an inward directed one after one loop of the parametrization of $f(\Gamma)$, so that in total all the contributions of the integration over the complete cycle $f(\Gamma)$  can  cancel out each other symmetrically when this happens.  
This phenomenon already occurs in the quaternionic setting, as pointed out in {\rm \cite{Fueter1948-49}, p. 199}.  
\end{remark}
\begin{remark}
As explained in {\rm \cite{Zoell}}, already in the quaternionic case there is no direct  correspondence anymore between the order of an $a$-point and the number of vanishing coefficients in the octonionic Taylor series expansion. Note that in complex analysis one has the relation 
$$
{\rm ord}(f-a;c) = n, \quad \Longleftrightarrow \quad (f-a)^{(k)}(c) = 0,\; \forall k < n,\;(f-a)^{(n)}(c) \neq 0. 
$$
Since the situation is already so complicated in the quaternions, it cannot be expected that one gets a simpler relation for the octonionic case. Actually, analogues of the counter-example presented in {\rm  \cite{Zoell}} on p.131-132 can easily be constructed.

 In the octonionic slice-regular setting described for instance in {\rm \cite{GPzeroes}}, the situation  is much simpler. As mentioned previously, in the slice function theoretical setting an octonionic slice-regular function either has isolated zeroes or spherical zeroes, similarly to the slice-monogenic setting in $\mathbb{R}^{n+1}$ cf. {\rm \cite{CSS,GPzeroes}}. In terms of the symmetric slice product the multiplicity of such a zero then can be described by the exponent of the (slice) power, namely in the usual way like in classical real and complex analysis: A slice-regular function $f$ can be decomposed uniquely in the way $f(z)=(z-a)^{*k}*g(z)$ where $g(z)$ is a uniquely defined and  zero-free slice-regular function around $a$, see {\rm \cite{CSS,GPzeroes}} and elsewhere. Note that ordinary powers of $z$ are intrinsic slice regular functions, also in the octonions. The slice-product gives some kind of symmetric structure. In the setting of $\mathbb{O}$-regular functions in the sense of the Cauchy-Riemann operator, such a decomposition is not possible, because of the lack of commutativity (and also of non-associativity).

\end{remark}

The definition of the order of an isolated $a$-point of an octonionic left or right $\mathbb{O}$-regular function in the sense of Definition~\ref{isolatedorder} is very natural from the topological point of view and so far the only meaningful tool to introduce a notion of ``multiplicity'' of an $a$-point. However, using this definition to calculate the value of the  order of a concrete practical example is very difficult in general. Note that one has to perform the integration over the {\em image} of the sphere. Now, a significant advantage of the octonionic setting in comparison to the Clifford analysis setting is that octonionic functions represent maps from $\mathbb{O} \to \mathbb{O}$ which can be uniquely identified with a map from $\mathbb{R}^8 \to \mathbb{R}^8$, by identifying the map $$x_0 + x_1 e_1 + \ldots + x_7 e_7 \mapsto f_0(z) + f_1(z) e_1 + \cdots + f_7 (z)e_7$$ with the corresponding map $\left(\begin{array}{c} x_0 \\ x_1\\ \vdots \\ x_7 \end{array}\right) \mapsto \left(\begin{array}{c} f_0(x_0,\ldots,x_7) \\ f_1(x_0,x_1,\ldots,x_7)\\ \vdots \\ f_7(x_0,x_1,\ldots,x_7) \end{array}\right)$. However, in Clifford analysis one deals with maps from $\mathbb{R}^8$ to $Cl_8 \cong \mathbb{R}^{128}$. Now, if the $7$-dimensional surface $\partial G$ is parametrized as in (\ref{param}), the image of that surface $f(\partial G)$ can be parametrized as
$$
f(\partial G) = \{f_0(x(\lambda_1,\ldots,\lambda_7)) + \sum\limits_{i=1}^7 f_i(x(\lambda_1,\ldots,\lambda_7)e_i)\}
$$
and one can simply apply the chain rule for ordinary real differentiable functions from $\mathbb{R}^8 \to \mathbb{R}^8$, as indicated in \cite{Hempfling} for purely paravector-valued functions. Applying the chain rule and exploiting the special mapping property that the image of octonionic functions are again octonions leads to the following octonionic generalization of the transformation formula from \cite{Kra2004} p. 32. In the Clifford analysis case one had to restrict onself to particular paravector-valued functions. This restriction is not necessary in the octonionic setting:
\begin{lemma}
	Let $G \subseteq \mathbb{O}$ be a domain and suppose that each real component function of an octonionic function $f:G \to \mathbb{O}$ is real differentiable in  the ordinary sense. Then we have 
	$$
	d\sigma(f(z)) = [(Jf)^{adj}(z)] \circledcirc [d\sigma(z)],
	$$
	where $[(Jf)^{adj}(z)]$ stands for the adjunct real component $8 \times 8$ matrix of the Jacobian $(Jf) = (\frac{\partial f_i}{\partial x_j})_{ij}$.Furthermore,  $[d\sigma(z)]$ represents the $\mathbb{R}^8$-vector composed by $\stackrel{\wedge}{dx_i}$ for $i=0,\ldots,7$ and $\circledcirc$ means the usual matrix-vector product, multiplying the real $8\times8$-matrix in the usual way with the $8$-dimensional real vector. The resulting $\mathbb{R}^8$-vector on the right-hand side then is re-interpreted as on octonion on the left-hand side identifying the unit vectors with the corresponding octonionic units.   
\end{lemma} 
It should be pointed out very clearly that $\circledcirc$ does not mean the usual octonionic product. To be more explicit $[d \sigma(z)]$ is interpreted as the vector 
$$
[d\sigma(z)] :=
\left(\begin{array}{c} (-1)^0  \stackrel{\wedge}{dx_0} \\ (-1)^1\stackrel{\wedge}{dx_1}\\ \vdots \\ (-1)^7  \stackrel{\wedge}{dx_7} \end{array}\right).$$ 
The adjunct matrix $[(Jf)^{adj}(z)]$ has the form 
$$
(Jf)^{adj}(z) =\Bigg((-1)^{i+j} \det \Big( \frac{\partial f_i}{\partial x_j}(z) \Big)^{adj}  \Bigg)_{i,j}.
$$
This also provides a correction to \cite{Kra2004} p. 32 where the index $i$ of the function $f$ has been forgotten as well as the star after $(Jf)$ (indicating the adjunct) in the second line of the proof. The proof for the octonionic case can be done along the same lines as presented for the paravector-valued Clifford case in \cite{Kra2004} p. 32. The chain rule leads to 
$$
d\sigma(f(z)) = \sum\limits_{i=0}^7 \sum\limits_{j=0}^7 (-1)^{i+j} e_i \det \Big(\frac{\partial f_i}{\partial x_j}(z)\Big)^{adj}(-1)^j \stackrel{\wedge}{dx_j}
$$
and the stated formula follows, because no associativity property is required. 
\par\medskip\par
This lemma allows us to reformulate the definition of the order given in Definition~\ref{isolatedorder} in the way that the integration is performed over the simple sphere $S_7(c,\varepsilon)$. In contrast to the Clifford analysis case presented in \cite{Kra2004} p. 33 we do not need to worry about a possible restriction of the range of values. All octonion-valued functions satisfying the left or right octonionic Cauchy-Riemann system are admitted here. 
However, the way how we put the brackets in the following theorem is crucially important. In the left $\mathbb{O}$-regular case we have
\begin{theorem}\label{order-reformulated}
Let $G \subseteq \mathbb{O}$ be a domain. Let $f:G \to \mathbb{O}$ be a left $\mathbb{O}$-regular function and suppose that $c \in G$ is an isolated $a$-point of $f$ with $f(c)=a$. Choose $\varepsilon > 0$ such that $\overline{B}(c,\varepsilon) \subseteq G$ and $f(z) \neq 0$ for all $z \in \overline{B}(c,\varepsilon) \backslash\{c\}$. Then the order of the $a$-point can be re-expressed by
\begin{eqnarray*}
{\rm ord}(f-a;c) &=& \frac{3}{\pi^4} \int\limits_{S_7(c,\varepsilon)} q_{\bf 0}(f(z)-a) \cdot  \Bigg(  [(Jf)^{adj}(z)] \circledcirc [d\sigma(z)] \Bigg)\\ &=&  \frac{3}{\pi^4} \int\limits_{S_7(c,\varepsilon)} \frac{\overline{f(z)-a}}{|f(z)-a|^8}\cdot  \Bigg(  [(Jf)^{adj}(z)] \circledcirc [d\sigma(z)] \Bigg).
\end{eqnarray*}	
\end{theorem}
Here, $\cdot$ stands for the octonionic product, where the term inside the large parenthesis on the right is re-interpreted as octonion.

Note that the Jacobian determinant is invariant under translations. Therefore $J(f-a)(z) = Jf(z)$. In the complex case the Jacobian simplifies to $(f-a)'(z)= f'(z)$ and one re-obtains the usual integrand $\frac{f'(z)}{f(z)-a}$ because the Cauchy kernel then coincides with the simple inverse. 
  
For the sake of completeness, in the right $\mathbb{O}$-regular case one obtains 
\begin{eqnarray*}
	{\rm ord}(f-a;c) &=& \frac{3}{\pi^4} \int\limits_{S_7(c,\varepsilon)}  \Bigg(  [(Jf)^{adj}(z)] \circledcirc [d\sigma(z)] \Bigg) \cdot q_{\bf 0}(f(z)-a) \\ &=&  \frac{3}{\pi^4} \int\limits_{S_7(c,\varepsilon)}    \Bigg(  [(Jf)^{adj}(z)] \circledcirc [d\sigma(z)] \Bigg) \cdot \frac{\overline{f(z)-a}}{|f(z)-a|^8}.
\end{eqnarray*}	
 
\par\medskip\par
Note that we always have ord$(f-a;c)=0$ in all points $c$ where $f(c) \neq a$. 
As a direct application this property and the statement of Theorem~\ref{order-reformulated} we can deduce the following argument principle for isolated $a$-points of $\mathbb{O}$-regular functions which provides an extension of Theorem 1.34 from \cite{Kra2004} where the paravector-valued Clifford holomorphic case has been treated. But also in the octonionic case we have 
\begin{theorem} (Octonionic argument principle)\\
Let $G \subseteq \mathbb{O}$ be a domain and suppose that $f:G \to \mathbb{O}$ is left $\mathbb{O}$-regular over $G$. Now, consider a nullhomologous $7$-dimensional cycle $\Gamma$ that parametrizes the surface of an $8$-dimensional oriented compact manifold $C \subset G$. Under the assumption that $f$ has only isolated $a$-points in the interior of $C$ and no further $a$-points on the boundary $\Gamma$, we have the order relation
$$
\sum\limits_{c \in C} {\rm ord}(f-a;c) = \frac{3}{\pi^4} \int\limits_{\Gamma} \frac{\overline{f(z)-a}}{|f(z)-a|^8}\cdot  \Bigg(  [(Jf)^{adj}(z)] \circledcirc [d\sigma(z)] \Bigg). 
$$
\end{theorem}  
\begin{proof}
	The proof follows along the same lines as in the Clifford analysis case given in \cite{Kra2004} pp.33. This is a consequence of its predominant topological nature. The crucial point is that any oriented compact manifold can have atmost finitely many isolated $a$-points in its interior, let us call them  $c_1,\ldots,c_n$. Thus, one can find a sufficiently small real number $\varepsilon > 0$ such that there are no $a$-points in the union of the sets $\bigcup_{i=1}^n B(c_i,\varepsilon) \backslash\{c_i\}$. Since $f$ has neither further $a$-points nor singular points in the remaining part $C \backslash \bigcup_{i=1}^n B_i$ one obtains in view of Theorem~\ref{order-reformulated} that 
	$$
	\int\limits_{\Gamma} \frac{\overline{f(z)-a}}{|f(z)-a|^8}\cdot  \Bigg(  [(Jf)^{adj}(z)] \circledcirc [d\sigma(z)] \Bigg) = \sum\limits_{i=1}^n \int\limits_{S(c_i,\varepsilon)} \frac{\overline{f(z)-a}}{|f(z)-a|^8}\cdot  \Bigg(  [(Jf)^{adj}(z)] \circledcirc [d\sigma(z)] \Bigg).
	$$
	The assertion now follows directly, when we take into account the mentioned property that ord$(f-a;c)=0$ at all $c \in C$ with $f(c) \neq a$.   
\end{proof}
The big goal of the argument principle is that it provides us with a toplogical tool to control the isolated $a$-points or zeroes of an octonionic regular function under special circumstances. Its classical application is Rouch\'e's theorem that presents a sufficient criterion to describe by which function an octonionic regular function may be distorted in the way that it has no influence on the numbers of isolated zeroes inside a domain, when particular requirements are met. Alternatively, it gives a criterion to decide whether two octonionic monogenic functions have the same number of isolated zeroes inside such a domain. In close analogy to the associative Clifford analysis case, cf. \cite{Kra2004} Theorem 1.35, we may establish 
\begin{theorem}\label{rouche} (Generalized classical Rouch\'e's theorem)\\
Suppose that $G \subseteq \mathbb{O}$ is a domain and that $\Gamma$ is a nullhomologous $7$-dimensional cycle parametrizing the boundary of an oriented compact $8$-dimensional manifold $C \subset G$. Let $f,g:G \to \mathbb{O}$ be two  $\mathbb{O}$-regular functions that have only a finite number of zeroes inside of int $C$ and no zeroes on $\Gamma$. Provided that $|f(z)-g(z)| < |f(z)|$ for all $z \in \Gamma$, then 
$$
\sum\limits_{c \in C} {\rm ord}(f;c) = \sum\limits_{c \in C} {\rm ord}(g,c).
$$	
\end{theorem}   
Also the nature of this theorem is predominantly topological. The topological aspects play a much more profound role than the function theoretical aspects, which nevertheless are also needed because the proof uses the argument principle involving the particular Cauchy-kernel of the octonionic Cauchy-Riemann system. Let us define a family of left $\mathbb{O}$-regular functions depending on a {\em continuous} real parameter $t \in [0,1]$ by 
$$
h_{t}(z) := f(z)+ t(g(z)-f(z)), \quad z \in G.
$$
For each $t \in [0,1]$ each function $h_z$ is left $\mathbb{O}$-regular over $G$, since $t$ is only a {\em real} parameter. Note that otherwise, the left $\mathbb{O}$-regularity would be destroyed in general. Let $z \in \Gamma$. Then we have $|t(g(z)-f(z)|=|t||f(z)-g(z)| \le |f(z)-g(z)| < |f(z)|$, where the latter inequality follows from the assumption. Therefore $h_t(z) \neq 0$ for all $z \in \Gamma$. 

 Furthermore, for each $t \in [0,1]$ the entity ord$(h_t;c)$ is an integer. Since the number of zeroes is supposed to be finite in $G$, for each $t$ the sum $\sum\limits_{c \in C} {\rm ord}(h_t;c)$ is finite and represents a finite integer $N(t) \in \mathbb{Z}$. Per definition we have 
\begin{eqnarray*}
	N(t) &=& \sum\limits_{c \in C} {\rm ord}(h_t;c)\\
	     & =& \frac{3}{\pi^4} \int\limits_{\Gamma} q_{\bf 0}(h_t(z)) \cdot \Bigg(
	     [(Jh_t)^{adj}(z)] \circledcirc [d\sigma(z)] 
	     \Bigg)\\
	     &=& \frac{3}{\pi^4} \int\limits_{\Gamma} q_{\bf 0}(f(z)+t g(z)-t f(z)) \cdot \Bigg(
	     [(J (f+tg-tf))^{adj}(z)] \circledcirc [d\sigma(z)] 
	     \Bigg).
\end{eqnarray*}
Since all terms under the latter integral are continuous functions in the variable $t$, also the expression $N(t)$ on the left-hand side must be  continuous in the variable $t$. However, $N(t)$ is an integer-valued expression for any $t \in [0,1]$.Therefore, $N(t)$ must be a constant expression, hence $N(0)= \sum\limits_{c \in C} {\rm ord}(h_0;c) = \sum\limits_{c \in C} {\rm ord}(f;c) $ and $N(1) = \sum\limits_{c \in C} {\rm ord}(h_1;c) = \sum\limits_{c \in C} {\rm ord}(g;c)$ must be equal. 
\par\medskip\par
As a nice application of Theorem~\ref{rouche} we can establish the following weakened version of Hurwitz' theorem. The following statement can also be carried over to the quaternionic monogenic setting and to the context of paravector-valued monogenic functions in Clifford algebras, for which this statement has not been established so far, at least as far as we know. We prove 
\begin{theorem}(Generalized Hurwitz theorem)\\
	Let $G \subset \mathbb{O}$ be a domain. Suppose that $f_n: G \to \mathbb{O}$ is a normally convergent sequence of $\mathbb{O}$-regular functions with $f_n(z) \neq 0$ at all $z \in G$ and for  each $n \in \mathbb{N}$. Then the limit function $f(z):= \lim\limits_{n \to \infty} f_n(z)$ has the property that either $\sum\limits_{c \in G} {\rm ord}(f;c)=0$ for all $c \in G$ or $f$ vanishes identically over $G$. 
 \end{theorem}
  
\begin{proof}
According to \cite{XL2002} Theorem 11, left (or right) $\mathbb{O}$-regular functions satisfy Weierstra{\ss}' convergence theorem. 
Therefore, the limit function $f$ is a well-defined $\mathbb{O}$-regular function over the whole domain $G$. Let us assume now that $f \not\equiv 0$ over $G$. Take an arbitrary point $z^* \in G$. In view of the identity theorem of left $\mathbb{O}$-regular functions (Proposition~\ref{identity}) there must exist a positive real $R > 0$ such that the closed ball $\overline{B(z^*,R)}$ is entirely contained inside $G$ and $M:= \min_{z \in S_7(z^*,R)} |f(z)| > 0$. Moreover, since $S_7(z^*,R)$ is compact there must exist an index $n_0 \in \mathbb{N}$ such that 
$$
\max_{z \in S_7(z^*,R)} |f(z)-f_n(z)| < M,\quad \forall n \ge n_0.
$$   
Summarizing, for all indices $n \ge n_0$, we have the inequality
$$
|f(z)-f_n(z)| < M \le |f(z)| \quad\quad \forall z \in S_7(z^*,R) 
$$	
which is the required condition of Rouch\'e's theorem in Theorem~\ref{rouche}. 

Now Rouch\'e's theorem tells us that 
$$
\sum\limits_{c \in S_7(z^*,R)} {\rm ord}(f;c) = \sum\limits_{c \in S_7(z^*,R)} \underbrace{{\rm ord}(f_n;c)}_{=0}.
$$
Note that since $f_n(z) \neq 0$ for all $z \in G$ we have ${\rm ord}(f_n;c)=0$ for all $n \in \mathbb{N}$. Since the points $z^*$ can be chosen arbitrarily inside of $G$, we can conclude that 
$$
\sum\limits_{c \in G} {\rm ord}(f;c)  = 0
$$
and the statement is proven.
\end{proof}
\begin{remark}
Note that in contrast to the complex analytic case, ord$(f;c) = 0$ does not guarantee that $f(c)\neq 0$, as pointed out in Remark~\ref{zero-order}. Therefore, we can only establish this weaker statement.
\end{remark}
\begin{remark}
In the context of other regularity concepts, such as for slice-regular octonionic functions and generalized octonionic holomorphic functions in the sense of S.V. Ludkovski, generalized statements of Rouch\'e and Hurwitz type could be established, see {\rm \cite{GPzeroes,L2007}}. 
\end{remark}

 \section{Rudiments for the treatment of non-isolated zeroes}
The following section presents results which are even new for quaternionic functions and paravector-valued functions in associative Clifford algebras. 

The aim is to present a meaningful definition of the order of zeroes or $a$-points of an $\mathbb{O}$-regular function that are not-isolated but lying on a $k$-dimensional simply connected compact manifold of dimension $1 \le k \le 6$, including in the simplest case compact algebraic varieties in eight variables.  

The case $k=0$ is the isolated case which has been treated in the previous section. As mentioned in the previous section, the case $k=7$ does not appear in the $\mathbb{O}$-regular setting, because of the  identity theorem for $\mathbb{O}$-regular functions (Proposition~\ref{identity}), which excludes this situation. Without loss of generality we focus on the treatment of compact varieties of zeroes, because varieties of $a$-points can be studied in the same way by looking at the function $f(z)-a$. 

Let us recall that in the isolated case one can always consider a small sphere around that zero with the property that no zeroes lie inside or on the boundary of that sphere. 

Let us now suppose that we have a $k$-dimensional simply-connected compact variety of zeroes ($k \le 6$), that we call $M$. To leave it simple we restrict ourselves in all that follows to those varieties that do not have auto-intersections.  

In the case of dealing with a variety of non-isolated zeroes with these properties,  
the proper analogue of a sphere surrounding an isolated point is a tubular domain of thickness $\varepsilon >0$ of the form
$$
T_M^{\varepsilon} :=\{z \in \mathbb{O} \backslash M \mid \min_{c \in M}\{|z-c|\} =\varepsilon\}.
$$   
In the case where $k = {\rm dim}\;M=1$ and where $M$ is a finite closed line segment, parametrized in the form $[\gamma] = \gamma(t),\quad t \in [0,1]$, the domain 
$$
T_{[\gamma]}^{\varepsilon} :=\{z \in \mathbb{O} \backslash M \mid \min_{t \in [0,1]}\{|z-\gamma(t)|\} =\varepsilon\}
$$ 
is nothing else than an ordinary symmetric circular tube of thickness $\varepsilon$ around that line segment. In the case where $M$ is a closed circle, the associated tubular domain $T_{[\gamma]}^{\varepsilon}$ is a generalized torus, more precisely it is homeomorphically equivalent to the real Hopf manifold $S_1 \times S_6 \cong \mathbb{R}^8 \backslash\{0\}/\mathbb{Z}$. A concrete example of a left and right $\mathbb{O}$-regular function where the zero set is up to atmost two isolated points  the unit circle lying in the subspace generated by $e_1$ and $e_2$ is the function $f(z)=Z_1^2+Z_2^2-1+\sum\limits_{j=3}^7 Z_j e_j$, where again $Z_i = x_i - x_0 e_i$ for all $i=1,\ldots,7$.   

In the particular case where $M$ is just an isolated point, say $M =\{z_0\}$, the tube then reduces to the set $T_{z_0} = \{z \in \mathbb{O} \mid |z-z_0| = \varepsilon\}$ which is the ordinary sphere $\partial B_8(z_0;\varepsilon)$ of the eight-dimensional ball. Thus, tubular domains provide us with a natural analogue for circular symmetric neighborhoods around closed simply connected manifolds with no auto-intersections.
 
In this framework, this is an adequate geometry to meaningfully introduce the notion of the order of a compact simply connected zero manifold of a left $\mathbb{O}$-regular function, generalizing the definitions given above for the isolated case. 
We introduce 
\begin{definition}
Suppose that $U \subseteq \mathbb{O}$ is a non-empty open set. Let $f:U \to \mathbb{O}$ be left $\mathbb{O}$-regular and suppose that $M$ is a compact simply connected manifold of dimension $k \in \{0,1\ldots,6\}$ with the above mentioned properties and with $M \subset U$ and $f(z)=0$ for all $z \in M$. Further assume that there is a real positive $\varepsilon > 0$ such that $T_M^{\varepsilon} \subset U$ and that $f(z) \neq 0$ for all $z \in T_M^{\varepsilon}$ and for all $z \in int T_M^{\varepsilon} \backslash M$.Then we can define the order of the non-isolated zero variety $M$ of $f$ by
\begin{eqnarray*}
	{\rm ord}(f;M) &:=& \frac{3}{\pi^4} \int\limits_{f(T_M^{\varepsilon})} q_{\bf 0}(w) d \sigma(w)\\
	&=& \frac{3}{\pi^4} \int\limits_{T_M^{\varepsilon}} q_{\bf 0}(f(z))  \cdot 
	\Bigg(  [(Jf)^{adj}(z)] \circledcirc [d\sigma(z)] \Bigg).
\end{eqnarray*}	 
\end{definition}   
\begin{remark}
	The integral counts how often the image of the tubular surface $T_M^{\varepsilon}$ under $f$ wraps around zero. 
	All zeroes belonging to the same zero variety $M$ have the same order, because the winding number is in view of its homotopic property a continuous and hence constant expression. The zero variety $M$ is simply-connected. Therefore ${\rm ord}(f;c_i) = {\rm ord}(f;c_j) = {\rm ord}(f;M)$ for all $c_i,c_j \in M$. 
\end{remark}
Notice further that the integral expressions are really well-defined because we do not integrate over any zeroes of $f$; $f(z) \neq 0$ for all $z \in  T_M^{\varepsilon}$.  

This generalized notion allows us to set up a generalized version of the octonionic argument principle where we now may admit left $\mathbb{O}$-regular functions having a finite number of compact simply-connected zero varieties $M_1,\ldots,M_p$ with no auto-intersections of dimension $k_1,\ldots,k_p$, respectively, lying inside a domain $G \subset \mathbb{O}$. We can prove
\begin{theorem} (Generalized octonionic argument principle for non-isolated zeroes)\\
	Let $G \subset \mathbb{O}$ be a domain. Suppose that $f:G \to \mathbb{O}$ is a left $\mathbb{O}$-regular function over $G$. Assume that $C$ is an $8$-dimensional oriented compact manifold $C \subset G$ whose boundary is parametrized by a $7$-dimensional null-homologous cycle $\Gamma$.  Furthermore, suppose that $f$ has a finite number of simply-connected closed zero varieties $M_1,\ldots,M_p$ with no auto-intersections of dimension $k_1,\ldots,k_p$, respectively, and that $f$ has no further zeroes inside of $C$ nor on its boundary $\Gamma$. Then we have 
	$$
	\frac{3}{\pi^4} \int\limits_{f(\Gamma)} q_{\bf 0}(w) d \sigma(w) = \sum\limits_{i=1}^p {\rm ord}(f;M_i).
	$$ 
\end{theorem}
\begin{proof}
Since $f$ has no zeroes on $\Gamma$ and since $C$ is compact the following integral and integral transformation is well-defined:

\begin{equation}\label{preveq}
\frac{3}{\pi^4}  \int\limits_{f(\Gamma)} q_{\bf 0}(w)d\sigma(w) = \frac{3}{\pi^4}  \int\limits_{\Gamma} q_{\bf 0}(f(z))  \cdot 
\Bigg(  [(Jf)^{adj}(z)] \circledcirc [d\sigma(z)] \Bigg).
\end{equation}
Since $f$ has no zeroes in $C \backslash \bigcup_{i=1}^p M_i$, we have that $\sum\limits_{c \in C \backslash \bigcup_{i=1}^p M_i} {\rm ord}(f;c) = 0$, so that   
the latter integral from (\ref{preveq}) can be expressed in the form 
$$
\frac{3}{\pi^4}\sum\limits_{i=1}^p\int\limits_{T_{M_i}^{\varepsilon_i}} q_{\bf 0}(f(z))  \cdot 
\Bigg(  [(Jf)^{adj}(z)] \circledcirc [d\sigma(z)] \Bigg) = \sum\limits_{i=1}^p {\rm ord}(f;M_i),
$$
because the contribution of this integral over the boundary of a domain that contains no zeroes inside is zero. 
\end{proof}
\begin{remark}
	The statement remains valid in the Clifford analysis setting, addressing paravector-valued functions with zero varieties that have the above mentioned properties. 
\end{remark}
\section{Perspectives}
The previous section suggests an approach how to address orders of non-isolated zeroes of octonionic regular or Clifford monogenic functions in the sense of the Riemann approach. A further step would consist in applying this argument principle to establish generalizations of Rouch\'e's theorem and Hurwitz' theorem to the non-isolated context. Obviously, the geometric conditions claimed in the previous section are very strong. As mentioned in Section~3 it is very easy to also construct $\mathbb{O}$-regular functions that have zero varieties of infinite extension. If we want to address  varieties with auto-intersections, then we have to adapt the use of tubular domains.  An important question is to investigate which genus do the arising zero manifolds have in the most general case. To get some insight in these kinds of questions a profound study of algebraic geometrical methods, in particular a deep study of understanding the nature of the appearing zero varities of $\mathbb{O}$-regular functions is required. Working on the intersection of algebraic geometry and hypercomplex function theories represents a promising branch for future investigation.

Furthermore, this paper shows that the argument principle is more a topological theorem than an analytic one, although the Cauchy kernel is explicitly needed in its definition. However, the predominant topological character gives the hope that these kinds of theorems can be carried over to many more hypercomplex function theories, in particular to the context of null-solutions to other differential equations. However, a really substantial question is to ask whether these tools can be carried over to functions that are defined in other algebras beyond octonions and paravector-valued subspaces of Clifford algebras. Both paravector spaces and octonions are normed and are free of zero-divisors. Following K. Imaeda \cite{Imaeda}, already in the context of sedenions it is not anymore possible to set up a direct analogue of Cauchy's integral formula. Cauchy's integral formula however is the basic tool for establishing all these results. The appearance of zero divisors will also have an impact on the topological properties. There remain a lot of open questions and challenges for future research.


\begin{thebibliography}{99}
 	
\bibitem{AH} P. Alexandroff, P. Hopf. {\it Topologie I}. Chelsea, Bronx, New York, 1935.

\bibitem{Baez} J. Baez. The octonions, {\it Bull. Amer. Math. Soc.} {\bf 39} (2002), 145--205. 

\bibitem{bds} F. Brackx, R. Delanghe, F. Sommen. {\it Clifford Analysis},
Pitman Res. Notes in Math., 76, 1982.

\bibitem{Burdik} C. Burdik, S. Catto, Y. G\"urcan, A. Khalfan, L. Kurt, V. Kato La. $SO(9,1)$ Group and Examples of Analytic Functions,  {\it Journal of Physics: Conf. Series} {\bf 1194} (2019), 012016.
  
 \bibitem{CDieckmann} C. Dieckmann. {\it Jacobiformen \"uber den Cayley-Zahlen}, PhD Thesis, Lehrstuhl A f\"ur Mathematik, 2014, https://publications.rwth-aachen.de/record/445009/files/5202.pdf
 
 \bibitem{CSS} F. Colombo, I. Sabadini, D. Struppa. {\it Entire slice regular functions}, Springer, 2016, Cham. 

\bibitem{DM} T. Dray, C. Manogue. {\it The Geometry of the Octonions}, World Scientific, Singapore, 2015, 228pp. https://doi.org/10.1142/8456 

\bibitem{Freitag} E. Freitag. {\it Complex Analysis 2}. Springer, Heidelberg, 2011.
 
\bibitem{Fueter1948-49} R. Fueter. {\it Functions of a Hypercomplex Variable}.  Lecture Notes written and supplemented by E. Bareiss, ETH Z\"urich, Wintersemester 1948/49.  

\bibitem{GPzeroes} R. Ghiloni, A. Perotti. Zeros of regular functions of quaternionic and octonionic variable: a division lemma and the camshaft effect. {\it Annali di Matematica Pura ed Applicata} {\bf 190} (2011), 539–-551.
 
\bibitem{GP} R. Ghiloni, A. Perotti.  Slice regular functions on real alternative	algebras. {\it Adv. Math.}, {\bf 226} (2011), 1662--1691. 

\bibitem{G} M. Gogberashvili. Octonionic geometry and conformal transformations, {\it International Journal of Geometric Methods in Modern Physics} {\bf 13}, No.07, 1650092 (2016) 

\bibitem{ghs} K. G\"urlebeck, K. Habetha, W. Spr\"o{\ss}ig. {\it Holomorphic Functions in the Plane and $n$-dimensional space}. Birkh\"auser, Basel, 2008.


\bibitem{GTBook} F. G\"ursey, H. Tze. {\it  On the role of division and Jordan algebras in particle physics}, World Scientific, Singapore, 1996.

\bibitem{Hempfling} T. Hempfling. Some Remarks on Zeroes of Monogenic Functions, {\it Adv. Appl. Clifford Algebras} 11(S2) (2001), 107--116.

\bibitem{HeKra} T. Hempfling, R.S. Krau{\ss}har. Order Theory of Isolated Points of Monogenic Functions. {\it Arch. Math.} {\bf 80} (2003), 406--423.
 
 \bibitem{Imaeda} K. Imaeda.  Sedenions: algebra and analysis. {\it Appl. Math. Comp.} {\bf 115} (2000), 77--88.
 
 \bibitem{JRS} Ming Jin, Guangbin Ren, I. Sabadini. Slice Dirac operator over octonions. To appear in {\it Israel J. Math.}, https://arxiv.org/pdf/1908.01383.pdf
 
\bibitem{KO2018} J. Kauhanen, H. Orelma. Cauchy-Riemann Operators in Octonionic Analysis, {\it Advances in Applied Clifford Algebras} {\bf 28} No. 1 (2018), 14pp.

\bibitem{KO2019} J. Kauhanen, H. Orelma. On the structure of Octonion regular functions, {\it Advances in Applied Clifford Algebras} {\bf 29} No. 4 (2019), 17pp. 

\bibitem{Kra2004} R.S. Krau{\ss}har. {\it Generalized automorphic Forms in hypercomplex spaces}. Birkh\"auser, Basel, 2004.

\bibitem{Kra2019-1} R.S. Krau{\ss}har. Function Theories in Cayley-Dickson algebras and Number Theory, submitted for publication (2019), 19pp,  https://arxiv.org/abs/1912.01351

\bibitem{Kra2019-2} R.S. Krau{\ss}har. Conformal mappings revisited in the octonions and Clifford algebras of arbitrary dimension, submitted for publication (2019), 12pp, 
https://arxiv.org/abs/1912.09109

\bibitem{L2007} S.V. Ludkovski. Differentiable functions of Cayley-Dickson numbers and line integration. {\it Journal of Mathematical Sciences} {\bf 141}, No. 3  (2007), 1231--1298.

% \bibitem{MD} C. Manogue, T. Dray. Octonionic M\"obius transformations, {\it Modern Physics Letters A} {\bf 14} (No. 19) (1999), 1243--1255. 

\bibitem{Nolder2018} C. Nolder. Much to do about octonions. {\it AIP Proceedings}, ICNAAM 2018, 2116, 160007 (2019); https://doi.org/10.1063/1.5114151
 

\bibitem{Nono} K. Nono. On the octonionic linearization of Laplacian and octonionic function theory, {\it Bull. Fukuoka Univ. Ed. Part} III {\bf 37} (1988), 1-–15.

% \bibitem{NSM} G. Najarbashi, B. Seifi, S. Mirzaei. Two- and three-qubit geometry, quaternionic and octonionic conformal maps, and intertwining stereographic projection,  {\it Quantum Inf Process} {\bf 15} (2016), 509-–528.
 
 
\bibitem{WarrenDSmith} W. D. Smith. Quaternions, octonions, 16-ons and $2^n$-ons; New kinds of numbers, Pensylvania State University (2004), 1--68. DOI: 10.1.1.672.2288

\bibitem{XL2000} Xing-min Li and Li-Zhong Peng.  On Stein-Weiss conjugate harmonic function and octonion analytic function, {\it Approx. Theory and its Appl.} {\bf 16} (2000), 28–-36.

\bibitem{XL2001} Xing-min Li, K. Zhao, Li-Zhong Peng, The Laurent series on the octonions. Adv. Appl.Clifford Alg. {\bf 11} (S2) (2001), 205-–217.

\bibitem{XL2002} Xing-min Li and Li-Zhong Peng. The Cauchy integral formulas on the octonions, {\it Bull. Belg. Math. Soc.} {\bf 9}   (2002), 47--62.  
 

\bibitem{XZL} Xing-min Li , Zhao Kai, Li-Zhong Peng. Characterization of octonionic analytic functions, {\it Complex Variables} {\bf 50} No. 13 (2005), 1031--1040. 

\bibitem{Zoell} G. Z\"oll. {\it Ein Residuenkalk\"ul in der Clifford-Analysis und die M\"obius Transformationen f\"ur euklidisce R\"aume}. PhD thesis, Lehrstuhl II f\"ur Mathematik, RWTH Aachen, 1987. 

\end{thebibliography}
\end{document}